\documentclass{amsproc}

\usepackage{amsmath, amsthm, amssymb, mathtools, mathrsfs, stmaryrd}
\usepackage{ascmac}
\usepackage{comment}
\usepackage{bm}
\allowdisplaybreaks

\usepackage{graphicx}
\usepackage[top=25mm, bottom=25mm, left=24mm, right=24mm]{geometry}

\usepackage{hyperref}

\usepackage{tikz}
\usetikzlibrary{intersections, calc, arrows.meta}
\usepackage{ytableau}

\usepackage{here}
\usepackage{time}
\usepackage[abbrev]{amsrefs}

\usepackage{xcolor}
\usepackage[capitalize,nameinlink,noabbrev,nosort]{cleveref}
\hypersetup{
	colorlinks=true,       
	linkcolor=brown,          
	citecolor=brown,        
	filecolor=brown,      
	urlcolor=brown,           
}

\makeatletter
\@namedef{subjclassname@2020}{\textup{2020} Mathematics Subject Classification}
\makeatother

\newtheorem*{theorem*}{\hspace{-6.3mm}\textbf{Theorem}}  

\newtheorem{theoremcounter}{Theorem Counter}[section]

\theoremstyle{remark}
\newtheorem{remark}{Remark}

\theoremstyle{definition}
\newtheorem{definition}[theoremcounter]{Definition}

\theoremstyle{plain}
\newtheorem{lemma}[theoremcounter]{Lemma}
\newtheorem{proposition}[theoremcounter]{Proposition}
\newtheorem{corollary}[theoremcounter]{Corollary}

\newtheorem{theorem}[theoremcounter]{Theorem}

\numberwithin{equation}{section}

\newcommand{\Z}{\mathbb{Z}}
\newcommand{\Q}{\mathbb{Q}}
\newcommand{\R}{\mathbb{R}}
\newcommand{\C}{\mathbb{C}}

\newcommand{\dd}{\mathrm{d}}
\newcommand{\bbH}{\mathbb{H}}

\newcommand{\calM}{\mathcal{M}}
\newcommand{\calS}{\mathcal{S}}

\DeclareMathOperator{\ImNew}{Im}
\renewcommand{\Im}{\ImNew}

\DeclareMathOperator{\SL}{SL}

\newcommand{\pmat}[1]{\begin{pmatrix}#1\end{pmatrix}}

\newcommand{\smat}[1]{\bigl(\begin{smallmatrix}#1\end{smallmatrix}\bigr)}

\begin{document}

\title[]{Quasi-modularity in MacMahon partition variants and prime detection} 

\author[]{Soon-Yi Kang}
\address{Department of Mathematics, Kangwon National University, Chuncheon, 200-701, Republic of Korea}
\email{sy2kang@kangwon.ac.kr}

\author[]{Toshiki Matsusaka}
\address{Faculty of Mathematics, Kyushu University, Motooka 744, Nishi-ku, Fukuoka 819-0395, Japan}
\email{matsusaka@math.kyushu-u.ac.jp}

\author[]{Gyucheol Shin}
\address{Department of Mathematics, Sungkyunkwan University, Suwon 16419, Republic of Korea}
\email{sgc7982@gmail.com}

\thanks{The first author was supported by the National Research Foundation of Korea (NRF) funded by the Ministry of Education (NRF-2022R1A2C1007188). The second author was supported by JSPS KAKENHI (JP21K18141 and JP24K16901) and by the MEXT Initiative through Kyushu University's Diversity and Super Global Training Program for Female and Young Faculty (SENTAN-Q). The third author was supported by Basic Science Research Program through the National Research Foundation of Korea (NRF) grant funded by the Ministry of Education (2019R1A6A1A10073079) and by the Korea government (MSIT, RS-2024-00348504).}




\maketitle


\begin{abstract} 
	Building on the results of Craig, van Ittersum, and Ono, we provide a refined understanding of MacMahon's partition functions and their variants, including their quasi-modular properties and new prime-detecting expressions.
\end{abstract}


\section{Introduction}

For positive integers $k$ and $n$, \emph{MacMahon's partition function}, (also known as \emph{MacMahon's sum-of-divisor}) $M_k(n)$ is defined by
\[
	M_k(n) \coloneqq \sum_{\substack{0 < m_1 < m_2 < \cdots < m_k \\ n = m_1 d_1 + m_2 d_2 + \cdots + m_k d_k}} d_1 d_2 \cdots d_k.
\]
This represents the sum of the products of the heights across all possible arrangements of $n$ squares into $k$ rectangular blocks with varying widths.

\begin{figure}[H]
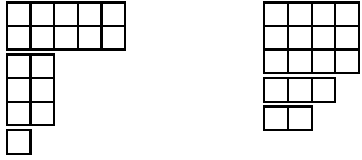

\centering
\ytableausetup{boxsize=3mm}
\begin{ytableau}
	\none & & & & &\\
	\none & & & & & \\
	\none \vspace{-2.5mm} \\
	\none & & \\
	\none & & \\
	\none & & \\
	\none \vspace{-2.5mm} \\
	\none & \\
\end{ytableau}
\qquad \qquad
\begin{ytableau}	
	\none & & & & \\
	\none & & & & \\
	\none & & & & \\
	\none \vspace{-2.5mm} \\
	\none & & & \\
	\none \vspace{-2.5mm} \\
	\none & & \\
\end{ytableau}
\caption{Two examples of partitions of 17 squares into 3 rectangular blocks: $1 \times 1$, $2 \times 3$, $5 \times 2$ and $2 \times 1$, $3 \times 1$, $4 \times 3$. In these cases, the products of the heights are given by $1 \cdot 3 \cdot 2$ and $1 \cdot 1 \cdot 3$, respectively.}
\end{figure}

As is clear from the definition, $M_1(n) = \sum_{d \mid n} d$ is the divisor sum. MacMahon's idea in~\cite{MacMahon1920} was to extend the divisor sums from the perspective of the theory of integer partitions. Recently, Craig--van Ittersum--Ono~\cite{CraigIttersumOno2024} revealed several prime-detecting expressions involving MacMahon's partition functions. For instance, they proved the following theorem:

\begin{theorem*}[Craig--van Ittersum--Ono]
	For positive integers $n$, we have
	\begin{align*}
		&(n^2 - 3n + 2) M_1(n) - 8M_2(n) \ge 0,\\
		&(3n^3 - 13n^2 + 18n - 8)M_1(n) + (12n^2 - 120n + 212) M_2(n) - 960M_3(n) \ge 0,
	\end{align*}
	and for $n \ge 2$, these expressions vanish if and only if $n$ is prime.
\end{theorem*}

They further discovered five distinct prime-detecting expressions involving MacMahon's partition functions (\cite[Theorem 1.2]{CraigIttersumOno2024}) and conjectured that any such expression (in terms of $M_k(n)$) can be written as a $\Q[n]$-linear combination of these five expressions. In this article, we present several questions that naturally arise from their results and provide answers to them. 

The motivation of the first question dates back to MacMahon's original article~\cite{MacMahon1920} from 1920. In this work, he introduced not only the $M_k(n)$ function discussed above but also several variants, which he examined in parallel. For instance, as a natural analogue of the generating function for $M_k(n)$, 
\begin{align}\label{eq:MacMahonA}
	A_k(q) \coloneqq \sum_{n=1}^\infty M_k(n) q^n = \sum_{0 < m_1 < m_2 < \cdots < m_k} \frac{q^{m_1 + m_2 + \cdots + m_k}}{(1-q^{m_1})^2 (1-q^{m_2})^2 \cdots (1- q^{m_k})^2},
\end{align}
he also introduced the following ``level 2" function:
\begin{align}\label{eq:MacMahonC}
	C_k(q) \coloneqq \sum_{n=1}^\infty M_k^{(2)}(n) q^n \coloneqq \sum_{\substack{0 < m_1 < m_2 < \cdots < m_k \\ m_i \not \equiv 0\ (2),\ (1 \le i \le k)}} \frac{q^{m_1 + m_2 + \cdots + m_k}}{(1-q^{m_1})^2 (1-q^{m_2})^2 \cdots (1- q^{m_k})^2}.
\end{align}
The function $C_k(q)$ often appears alongside $A_k(q)$ in the works of Andrews--Rose~\cite{AndrewsRose2013}, Bachmann~\cite{Bachmann2024}, Ono--Singh~\cite{OnoSingh2024}, and others. This naturally raises the question of whether $M_k^{(2)}(n)$ can also give rise to a prime-detecting expression. Our initial answer to this question leads to a more refined expression that detects not only primes but also powers of $2$ simultaneously.

\begin{theorem}\label{thm:level2}
	For positive integers $n \ge 2$, we have
	\[
		(n^2 - 4n + 3) M_1^{(2)}(n) - 24 M_2^{(2)}(n) \begin{cases}
			=0 &\text{if $n$ is odd prime},\\
			<0 &\text{if $n = 2^l$ for $l \in \Z_{\ge 1}$},\\
			>0 &\text{otherwise}
		\end{cases}
	\]
	and
	\begin{align*}
		&(n^4 - n^3 - 14n^2 + 29n - 15) M_1^{(2)}(n) - 120(3n-8) M_2^{(2)}(n) - 5760 M_3^{(2)}(n) \\
		&\begin{cases}
			=0 &\text{if $n$ is odd prime},\\
			<0 &\text{if $n = 2^l$ for $l \in \Z_{\ge 1}$},\\
			>0 &\text{otherwise}.
		\end{cases}
	\end{align*}
\end{theorem}

The proof follows similarly to that in \cite{CraigIttersumOno2024}, relying on the quasi-modularity of $C_k(q)$ on the level 2 congruence subgroup $\Gamma_0(2)$. A quick review of the theory of quasi-modular forms will be provided in \cref{section:Quasi}.

To reveal the quasi-modularity inherent in all variants studied by MacMahon, we develop a unified framework for MacMahon functions. Fix a positive integer $N$ and a non-empty subset $S \subset \Z/N\Z$. For positive integers $k$, we define the \emph{generalized MacMahon functions} as
\begin{align}\label{eq:Rose-generalized}
	A_{S,N,k}(q) \coloneqq \sum_{\substack{0 < m_1 < m_2 < \cdots < m_k \\ m_i \in S,\ (1 \le i \le k)}} \frac{q^{m_1 + m_2 + \cdots + m_k}}{(1-q^{m_1})^2 (1-q^{m_2})^2 \cdots (1-q^{m_k})^2},
\end{align}
where $m \in S$ means that when $m \in \Z$ is projected onto $\Z/N\Z$, its image lies in $S$. For instance, the above functions $A_k(q)$ and $C_k(q)$ are included as special cases where $A_k(q) = A_{\{0\}, 1, k}(q)$ and $C_k(q) = A_{\{1\}, 2, k}(q)$, respectively. This generalization was studied by Rose~\cite{Rose2015} in 2015, who showed that if $S$ is \emph{symmetric}, meaning $l \in S$ implies $-l \in S$, then $A_{S,N,k}(q)$ is a quasi-modular form for a certain congruence subgroup $\Gamma$. The precise identification of $\Gamma$ was later made by Larson~\cite{Larson2015}. As can be inferred from the fact that two articles are required to specify $\Gamma$, Rose's results are rather complicated. At least, they are not presented in a way that immediately yields basic equations such as
\begin{align}\label{eq:MacMahon-basics}
	M_1(n) = \sum_{d \mid n} d \quad \text{ or } \quad M_1^{(2)}(n) = \sum_{\substack{d \mid n \\ n/d \equiv 1\ (2)}} d.
\end{align}
Here, we aim to provide an alternative proof that more clearly shows the quasi-modularity.

\begin{theorem}\label{thm:MacMahon-Lehmer}
	Let $\Lambda_k(x_2, x_4, \dots, x_{2k}) \in \Q[x_2, x_4, \dots, x_{2k}]$ be a polynomial defined in \cref{def:Lehmer-poly}. Then, for any positive integer $N$ and a non-empty subset $S \subset \Z/N\Z$, we have
	\[
		A_{S,N,k}(q) = \Lambda_k(G_{S,N,2}(q), G_{S,N,4}(q), \dots, G_{S,N,2k}(q)),
	\]
	where we define
	\[
		G_{S,N,k}(q) \coloneqq \sum_{n=1}^\infty \bigg(\sum_{\substack{d \mid n \\ n/d \in S}} d^{k-1} \bigg)q^n.
	\]
	In particular, if $S$ is symmetric, then $A_{S,N,k}(q)$ is a (mixed weight) quasi-modular form of level $N$.
\end{theorem}

This formulation shows that the \emph{type} (i.e., the polynomial $\Lambda_k$) is determined independently of the choice of $N$ and $S$, and that the quasi-modularity of the series $G_{S,N,k}(q)$ immediately implies the quasi-modularity of $A_{S,N,k}(q)$. In addition, since $\Lambda_1(x_2) = x_2$, the equations in \eqref{eq:MacMahon-basics} follow directly. As shown in \cref{prop:Eisenstein-modular}, when $S$ is symmetric, the quasi-modularity of $G_{S,N,k}(q)$ is well-known, and thus the quasi-modularity of $A_{S,N,k}(q)$ becomes clear. The generalized MacMahon functions defined in \eqref{eq:Rose-generalized}, however, are not yet sufficiently generalized to encompass all the variants introduced by MacMahon himself. Further generalizations, along with the proof of \cref{thm:MacMahon-Lehmer}, are provided in \cref{sec:MacMahon-variants}. In \cref{section:level2}, we then prove \cref{thm:level2} using this quasi-modularity as a key step and discuss the prospects for further prime-detecting expressions.

Finally, while $M_1(n)$ inherently involves the structure of $n$'s divisors, we present prime-detecting expressions that rely exclusively on the counting of lattice points, without any dependence on prime factorization. Here, we consider the following three lattices (with a translation).

\begin{align*}
	L_1 &\coloneqq \Z \pmat{2 \\ 0 \\ 0 \\ 0} + \Z \pmat{0 \\ 2 \\ 0 \\ 0} + \Z \pmat{0 \\ 0 \\ \sqrt{2} \\ 0} + \Z \pmat{0 \\ 0 \\ 0 \\ \sqrt{2}} + \frac{1}{\sqrt{2}} \pmat{0 \\ 0 \\ 1 \\ 1},\\
	L_2 &\coloneqq \Z \pmat{2 \\ 0 \\ 0 \\ 0} + \Z \pmat{0 \\ 2 \\ 0 \\ 0} + \Z \pmat{0 \\ 0 \\ \sqrt{2} \\ 0} + \Z \pmat{0 \\ 0 \\ 0 \\ \sqrt{2}} + \frac{1}{\sqrt{2}} \pmat{\sqrt{2} \\ \sqrt{2} \\ 1 \\ 1},\\
	E_8 &\coloneqq \Z \pmat{2 \\ 0 \\ 0 \\ 0 \\ 0 \\ 0 \\ 0 \\ 0} + \Z \pmat{-1 \\ 1 \\ 0 \\ 0 \\ 0 \\ 0\\ 0 \\ 0} + \Z \pmat{0 \\ -1 \\ 1 \\ 0 \\ 0 \\ 0 \\ 0 \\ 0} + \Z \pmat{0 \\ 0 \\ -1 \\ 1 \\ 0 \\ 0 \\ 0 \\ 0} + \Z \pmat{0 \\ 0 \\ 0 \\ -1 \\ 1 \\ 0 \\ 0 \\ 0} + \Z \pmat{0 \\ 0 \\ 0 \\ 0 \\ -1 \\ 1 \\ 0 \\ 0} + \Z \pmat{0 \\ 0 \\ 0 \\ 0 \\ 0 \\-1 \\ 1 \\ 0} + \Z \pmat{1/2 \\ 1/2 \\ 1/2 \\ 1/2 \\ 1/2 \\ 1/2 \\ 1/2 \\ 1/2}.
\end{align*}
Here, $L_1$ and $L_2$ are translations of rectangular lattices, while the last lattice is the $E_8$-lattice, a special lattice that achieves the densest sphere packing in dimension 8~\cite{Viazovska2017}. For each lattice $L \in \{L_1, L_2, E_8\}$, we define
\[
	r_L(n) \coloneqq \#\{x \in L : \|x\|^2 = n\},
\]
where $\|{}^t (x_1, \dots, x_r)\|^2 = x_1^2 + \cdots + x_r^2$. Then, the following prime-detecting expressions hold.

\begin{theorem}\label{thm:lattice}
	For positive integers $n \ge 2$, we have
	\begin{align*}
		60(n^2-n+1) r_{L_1}(n) - r_{E_8}(2n) \begin{cases}
			=0 &\text{if $n$ is prime with $n \equiv 1 \pmod{4}$},\\
			< 0 &\text{if } n \not\equiv 1 \pmod{4},\\
			>0 &\text{otherwise}
		\end{cases}
	\end{align*}
	and
	\begin{align*}
		60(n^2-n+1) r_{L_2}(n) - r_{E_8}(2n) \begin{cases}
			=0 &\text{if $n$ is prime with $n \equiv 3 \pmod{4}$},\\
			< 0 &\text{if } n \not\equiv 3 \pmod{4},\\
			>0 &\text{otherwise}.
		\end{cases}
	\end{align*}
\end{theorem}

This provides refined prime-detecting expressions with respect to the remainder modulo $4$. To reveal the tricks, the idea of the proof is based on that for MacMahon's partition functions and more recent work by Gomez~\cite{Gomez2024} and Craig~\cite{Craig2024}. Recall that there are two standard approaches to constructing modular forms: one using Eisenstein series and the other using theta functions. The proof in \cite{CraigIttersumOno2024} employs Eisenstein series, and the result here is essentially a reformulation in terms of theta functions. While the modularity of Eisenstein series follows directly from its definition, establishing the modularity of theta functions relies on Poisson's summation formula in Fourier analysis, which is central to the nontriviality of the argument. This proof will be presented in \cref{section:theta}. Finally, although this article does not explore the approach, it is worth noting that Craig--van Ittersum--Ono~\cite{CraigIttersumOno2024} successfully constructed infinitely many prime-detecting expressions by considering a further generalization of the level 1 MacMahon's partition function, known as the MacMahonesque partition functions. This approach involves considering sum of products of powers of divisors, which is also a natural method of generalization.

\section{Quasi-modular forms}\label{section:Quasi}

First, we discuss a general theory of quasi-modular forms. Let $\bbH \coloneqq \{z \in \C : \Im(z) > 0\}$ denote the \emph{upper half-plane}. For a positive integer $N$, the subgroups $\Gamma_0(N)$ and $\Gamma_1(N)$ of $\SL_2(\Z)$, known as \emph{congruence subgroups}, are defined by
\begin{align*}
	\Gamma_0(N) &\coloneqq \left\{ \pmat{a & b \\ c & d} \in \SL_2(\Z) : c \equiv 0 \pmod{N}\right\},\\
	\Gamma_1(N) &\coloneqq \left\{ \pmat{a & b \\ c & d} \in \Gamma_0(N) : a, d \equiv 1 \pmod{N}\right\}.
\end{align*}
In this section, $\Gamma$ will refer to either $\Gamma_0(N)$ or $\Gamma_1(N)$. For an integer $k$, a holomorphic function $f: \bbH \to \C$ is called a \emph{modular form} of weight $k$ on $\Gamma$ if it satisfies the transformation law
\[
	f \left(\frac{az+b}{cz+d}\right) = (cz+d)^k f(z)
\]
for all $\smat{a & b \\ c & d} \in \Gamma$ and is also holomorphic at cusps $\Q \cup \{i\infty\}$. In particular, when a modular form $f$ decays exponentially at each cusp, it is called a \emph{cusp form}. (For a detailed definition, see~\cite[Chapter 1]{DiamondShurman2005}). Let $\calM_k(\Gamma)$ (resp.~$\calS_k(\Gamma)$) denote the $\C$-vector space of weight $k$ modular forms (resp.~cusp forms) on $\Gamma$. It is known that $\calM_k(\Gamma) = \{0\}$ for $k < 0$ and that $\calM_k(\Gamma)$ is finite-dimensional for any integer $k$. The most fundamental example is a class of functions called \emph{Eisenstein series}. For an even integer $k \ge 4$, we have
\begin{align}\label{eq:Eisenstein}
	E_k(z) \coloneqq 1 - \frac{2k}{B_k} \sum_{n=1}^\infty \sigma_{k-1}(n) q^n \in \calM_k(\SL_2(\Z)),
\end{align}
where $B_k$ is the $k$-th Bernoulli number, $\sigma_{k-1}(n) \coloneqq \sum_{d \mid n} d^{k-1}$, and $q = e^{2\pi iz}$. The structure of the space of modular forms on $\SL_2(\Z)$ is well understood. It is known that
\[
	\bigoplus_{k \in \Z} \calM_k(\SL_2(\Z)) = \C [E_4, E_6],
\]
implying every modular form can be expressed as a polynomial in $E_4(z)$ and $E_6(z)$. On the other hand, the function
\[
	E_2(z) \coloneqq 1 - 24 \sum_{n=1}^\infty \sigma_1(n) q^n,
\]
obtained by applying equation~\eqref{eq:Eisenstein} for $k = 2$, is not a modular form. However, the algebra 
\[
	\widetilde{\calM}(\SL_2(\Z)) \coloneqq \C[E_2, E_4, E_6],
\]
which extends the previous polynomial ring by including $E_2(z)$, satisfies the special property of being closed under the action of the $q$-differential operator
\[
	D \coloneqq \frac{1}{2\pi i} \frac{\dd}{\dd z} = q \frac{\dd}{\dd q}.
\]
This can be seen from Ramanujan's identities:
\begin{align}\label{eq:Ramanujan-id}
	DE_2 = \frac{E_2^2 - E_4}{12}, \quad DE_4 = \frac{E_2 E_4 - E_6}{3}, \quad DE_6 = \frac{E_2 E_6 - E_4^2}{2}.
\end{align}
We call an element of $\widetilde{\calM}(\SL_2(\Z)) = \C[E_2, E_4, E_6]$ a (mixed weight) \emph{quasi-modular form} on $\SL_2(\Z)$. Although quasi-modular forms are typically defined through (quasi-)modular transformation laws, we adopt this ad-hoc definition for convenience. For a comprehensive and intrinsic treatment, see the work of Kaneko--Zagier~\cite{KanekoZagier1995} or \cite[Section 5]{Zagier2008}. The space of quasi-modular forms on a congruence subgroup is defined analogously, and the following structure theorem holds~\cite[Proposition 1]{KanekoZagier1995}. This characterization will serve as our working definition throughout this article.
\begin{definition}
	Let $\Gamma$ be the congruence subgroup $\Gamma_0(N)$ or $\Gamma_1(N)$. We define a (mixed weight) \emph{quasi-modular form} on $\Gamma$ as an element of
	\begin{align}\label{eq:def-quasi-N}
		\widetilde{\calM}(\Gamma) \coloneqq \left(\bigoplus_{k \in \Z} \calM_k(\Gamma)\right) \otimes \C[E_2].
	\end{align}
	In particular, if it can be expressed in the homogeneous form
	\[
		f(z) = \sum_{0 \le j \le k/2} f_j(z) E_2(z)^j, \quad (f_j \in \calM_{k-2j}(\Gamma)),
	\]
	we call $f$ a quasi-modular form of weight $k$ on $\Gamma$.
\end{definition}

Let $\widetilde{\calM}_k(\Gamma)$ denote the $\C$-vector space of weight $k$ quasi-modular forms on $\Gamma$. The following structure theorem is also known.

\begin{proposition}[{\cite[Proposition 1]{KanekoZagier1995} and \cite[Proposition 20]{Zagier2008}}]\label{prop:quasi-structure}
	For any even integer $k \ge 2$, we have $D(\widetilde{\calM}_k(\Gamma)) \subset \widetilde{\calM}_{k+2}(\Gamma)$, and
	\[
		\widetilde{\calM}_k(\Gamma) = \bigoplus_{j=0}^{k/2-1} D^j (\calM_{k-2j}(\Gamma)) \oplus \C \cdot D^{k/2-1} \phi,
	\]
	where $\phi \in \widetilde{\calM}_2(\Gamma) \setminus \calM_2(\Gamma)$ is a weight $2$ non-modular quasi-modular form on $\Gamma$. (For instance, we can take $\phi = E_2$).
\end{proposition}

\section{MacMahon partition variants}\label{sec:MacMahon-variants}

\subsection{Generalized MacMahon functions}

MacMahon~\cite{MacMahon1920} developed his discussion further by introducing not only $A_k(q)$ and $C_k(q)$ defined in \eqref{eq:MacMahonA} and \eqref{eq:MacMahonC}, but also their variants $A_k(q), B_k(q), \dots, H_k(q)$. In this context, we introduce \emph{generalized MacMahon functions}, which encompass all of these as special cases, and examine their quasi-modularity.

\begin{definition}\label{def:ASNk}
	Fix a positive integer $N$ and a non-empty subset $S \subset \Z/N\Z$. For a positive integer $k$ and $\epsilon \in \{-1, 1\}$, we define
	\[
		A_{S, N, \epsilon, k}(q) \coloneqq \sum_{\substack{0 < m_1 < m_2 < \cdots < m_k \\ m_i \in S,\ (1 \le i \le k)}} \frac{\epsilon^k q^{m_1 + m_2 + \cdots + m_k}}{(1 - \epsilon q^{m_1})^2 (1- \epsilon q^{m_2})^2 \cdots (1 - \epsilon q^{m_k})^2}.
	\]
\end{definition}

This provides a further generalization of \eqref{eq:Rose-generalized}. MacMahon's variants are realized as
\begin{align*}
	A_k(q) &= A_{\{0\}, 1, 1, k}(q), & B_k(q) &= (-1)^k A_{\{0\}, 1, -1, k}(q),\\
	C_k(q) &= A_{\{1\}, 2, 1, k}(q), & D_k(q) &= (-1)^k A_{\{1\}, 2, -1, k}(q),\\
	E_k(q) &= A_{\{1,4\}, 5, 1, k}(q), & F_k(q) &= (-1)^k A_{\{1,4\}, 5, -1, k}(q),\\
	G_k(q) &= A_{\{2,3\}, 5, 1, k}(q), & H_k(q) &= (-1)^k A_{\{2,3\}, 5, -1, k}(q).
\end{align*}
Andrews--Rose~\cite{AndrewsRose2013} showed that $A_k(q)$ is a (mixed weight) quasi-modular form. Although not explicitly stated, the quasi-modularity of $C_k(q)$ can also be deduced from \cite[Corollary 3]{AndrewsRose2013}. Rose~\cite{Rose2015} later introduced the aforementioned generalization and showed that if the set $S$ is symmetric, then $A_{S, N, \epsilon, k}(q)$ is a quasi-modular form for a certain congruence subgroup $\Gamma$. The precise characterization of $\Gamma$ for the case $\epsilon =1$ was subsequently provided by Larson~\cite{Larson2015}. 

In this section, we provide an alternative proof of Rose's result by presenting a generating function for the generalized MacMahon functions $A_{S,N,\epsilon,k}(q)$ that differs from Rose's. This new generating function immediately shows that $A_{S,N,\epsilon,k}(q)$ can be expressed as a polynomial in Eisenstein series. Consequently, the quasi-modularity of Eisenstein series directly implies the quasi-modularity of $A_{S,N,\epsilon,k}(q)$. This contrasts with Rose's generating functions and Larson's work, where determining the group $\Gamma$ required additional efforts. To achieve this, we recall (normalized) Lehmer's polynomials $\Lambda_k(x_2, x_4, \dots, x_{2k})$, which were initially introduced experimentally in~\cite{Lehmer1966} through the first few terms, without a definition, and were later defined by the second author and Shibukawa~\cite{MatsusakaShibukawa2024} in the context of derivative values of cyclotomic polynomials.

\begin{definition}\label{def:Lehmer-poly}
	We define the polynomial $\Lambda_k(x_2, x_4, \dots, x_{2k}) \in \Q[x_2, x_4, \dots, x_{2k}]$ by the generating series
	\[
		1 + \sum_{k=1}^\infty \Lambda_k(x_2, x_4, \dots, x_{2k}) X^{2k} \coloneqq \exp \left(2 \sum_{j=1}^\infty \frac{(-1)^{j-1}}{(2j)!} \left(2 \arcsin \frac{X}{2}\right)^{2j} x_{2j}\right).
	\]
\end{definition}

The first few examples of the polynomials $\Lambda_k$ are given by
\begin{align*}
	\Lambda_1(x_2) &= x_2,\\
	\Lambda_2(x_2, x_4) &= \frac{1}{12} \bigg(6x_2^2 + x_2 - x_4 \bigg),\\
	\Lambda_3(x_2, x_4, x_6) &= \frac{1}{360} \bigg(60x_2^3 + 30x_2^2 - 2(15x_4 - 2) x_2 - 5x_4 + x_6 \bigg),\\
	\Lambda_4(x_2, x_4, x_6, x_8) &= \frac{1}{20160} \bigg(840x_2^4 + 840 x_2^3 -42(20x_4 - 7)x_2^2\\
		&\qquad + 4(14x_6 - 105x_4 + 9)x_2 + 70 x_4^2 - 49 x_4 + 
 14 x_6 - x_8 \bigg).
\end{align*}

\begin{remark}
	The original Lehmer polynomial $\Omega_k(x_2, \dots, x_{2k})$ is given by
	\[
		\Omega_k(x_2, x_4, \dots, x_{2k}) = (-1)^k \frac{(2k)!}{2B_{2k}} \Lambda_k(-B_2 x_2, -B_4 x_4, \dots, -B_{2k} x_{2k}).
	\]
\end{remark}

The main result presented here is that generalized MacMahon functions $A_{S,N,\epsilon,k}(q)$ can be expressed using this polynomial and the Eisenstein series defined below.

\begin{definition}
	Under the same notation as in \cref{def:ASNk}, we define the \emph{divisor sum} as
	\begin{align*}
		\sigma_{S,N,\epsilon, k}(n) &\coloneqq \sum_{\substack{d \mid n \\ n/d \in S}} \epsilon^d d^k.
	\end{align*}
	The associated \emph{Eisenstein series} is then defined by
	\begin{align*}
		G_{S, N, \epsilon, k}(q) \coloneqq \sum_{n=1}^\infty \sigma_{S,N,\epsilon,k-1}(n) q^n.
	\end{align*}
\end{definition}

Under these notations, for any non-empty subset $S \subset \Z/N\Z$, which is not necessarily symmetric, the following generating function is obtained.

\begin{theorem}[A generalized version of \cref{thm:MacMahon-Lehmer}]\label{thm:ASNk-poly}
	We have
	\[
		1 + \sum_{k=1}^\infty A_{S,N,\epsilon,k}(q) X^{2k} = \exp \left(2 \sum_{j=1}^\infty \frac{(-1)^{j-1}}{(2j)!} G_{S,N,\epsilon, 2j}(q) \left(2 \arcsin \frac{X}{2}\right)^{2j}\right),
	\]
	which implies that
	\[
		A_{S,N,\epsilon, k}(q) = \Lambda_k(G_{S,N,\epsilon, 2}(q), G_{S,N,\epsilon,4}(q), \dots, G_{S,N,\epsilon,2k}(q)).
	\]
\end{theorem}

\begin{proof}
	As MacMahon~\cite{MacMahon1920} showed in the simplest case, and as is clear from the definition of $A_{S,N,\epsilon,k}(q)$, the following holds:
	\begin{align*}
		1 + \sum_{k=1}^\infty A_{S,N,\epsilon,k}(q) (2\sin \theta)^{2k} = \prod_{0 < m \in S} \left(1 + \frac{\epsilon q^m}{(1- \epsilon q^m)^2} (2\sin \theta)^2 \right),
	\end{align*}
	where $0 < m \in S$ indicates that $m$ runs over the positive integers whose images in $\Z/N\Z$ lie in $S$. Taking the logarithm of both sides yields,
	\[
		\log \left(1 + \sum_{k=1}^\infty A_{S,N,\epsilon,k}(q) (2\sin \theta)^{2k} \right) = - \sum_{0 < m \in S} \sum_{n=1}^\infty \frac{(-1)^n}{n} \frac{\epsilon^n q^{mn}}{(1-\epsilon q^m)^{2n}} (2\sin \theta)^{2n}.
	\]
	This function is an even function (in $\theta$) with period $\pi$, implying that it has a Fourier series expansion of the form
	\begin{align}\label{eq:log-FourierN}
		- \sum_{0 < m \in S} \sum_{n=1}^\infty \frac{(-1)^n}{n} \frac{\epsilon^n q^{mn}}{(1-\epsilon q^m)^{2n}} (2\sin \theta)^{2n} = a_0 + \sum_{k=1}^\infty a_k \cos (2k \theta).
	\end{align}
	Since the following formula is known:
	\[
		(2\sin \theta)^{2n} = \binom{2n}{n} + 2 \sum_{k=1}^n (-1)^k \binom{2n}{n-k} \cos (2k \theta),
	\]
	(see, for instance,~\cite[Appendix to Chapter XI, p.471]{MOS1966}), the left-hand side of \eqref{eq:log-FourierN} can be expressed as
	\[
		- \sum_{0 < m \in S} \sum_{n=1}^\infty \frac{(-1)^n}{n} \frac{\epsilon^n q^{mn}}{(1-\epsilon q^m)^{2n}} \binom{2n}{n} - 2\sum_{0<m \in S} \sum_{n=1}^\infty \frac{(-1)^n}{n} \frac{\epsilon^n q^{mn}}{(1-\epsilon q^m)^{2n}} \sum_{k=1}^n (-1)^k \binom{2n}{n-k} \cos (2k \theta).
	\]
	
	First, to compute the constant term $a_0$, note that
	\[
		-\sum_{n=1}^\infty \frac{(-1)^n}{n} \binom{2n}{n} X^n = 2 \log \frac{1 + \sqrt{1+4X}}{2},
	\]
	which can be derived from the binomial theorem $(1-4X)^{-1/2} = \sum_{n=0}^\infty \binom{2n}{n} X^n$. The constant term is therefore
	\begin{align}\label{eq:constant-term}
	\begin{split}
		a_0 &= - \sum_{0 < m \in S} \sum_{n=1}^\infty \frac{(-1)^n}{n} \binom{2n}{n} \left(\frac{\epsilon q^m}{(1-\epsilon q^m)^2}\right)^n = -2 \sum_{0 < m \in S} \log(1-\epsilon q^m)\\
			&= 2\sum_{0<m \in S} \sum_{l=1}^\infty \frac{\epsilon^l q^{lm}}{l}.
	\end{split}
	\end{align}
	
	Next, for $k \ge 1$, we prove the equation
	\begin{align}\label{eq:Fourier-positive}
		a_k = -2 \sum_{0<m \in S} \sum_{n=k}^\infty \frac{(-1)^{n-k}}{n} \frac{\epsilon^n q^{mn}}{(1-\epsilon q^m)^{2n}} \binom{2n}{n-k} = -2 \sum_{0 < m \in S} \frac{\epsilon^k q^{km}}{k}.
	\end{align}
	As a refinement, we now show that
	\begin{align}\label{eq:refinement}
		\sum_{n=k}^\infty \frac{(-1)^{n-k}}{n} \frac{q^n}{(1-q)^{2n}} \binom{2n}{n-k} = \frac{q^k}{k}.
	\end{align}
	Replacing $q$ with $\epsilon q^m$ and taking the sum of both sides over $0 < m \in S$, we obtain \eqref{eq:Fourier-positive}. By applying the binomial expansion to $(1-q)^{-2n}$, the left-hand side of \eqref{eq:refinement} becomes
	\begin{align*}
		\sum_{n=k}^\infty \frac{(-1)^{n-k}}{n} \frac{q^n}{(1-q)^{2n}} \binom{2n}{n-k} &= \sum_{n=k}^\infty \frac{(-1)^{n-k}}{n} \binom{2n}{n-k} \sum_{l=0}^\infty \binom{l+2n-1}{2n-1} q^{n+l}\\
			&= \frac{q^k}{k} + \sum_{N=k+1}^\infty \left(\sum_{n=k}^N \frac{(-1)^{n-k}}{n} \binom{2n}{n-k} \binom{N+n-1}{2n-1} \right) q^N.
	\end{align*}
	We show that the coefficient of $q^N$ vanishes when $N > k$. By a direct calculation, we obtain
	\begin{align}\label{eq:Nth-coeff-N}
		\sum_{n=k}^N \frac{(-1)^{n-k}}{n} \binom{2n}{n-k} \binom{N+n-1}{2n-1} &= \frac{2}{N-k} \sum_{n=k}^N (-1)^{n-k} \binom{N-k}{n-k} \binom{N+n-1}{N-k-1}.
	\end{align}
	By multiplying two binomial expansions
	\begin{align*}
		\frac{1}{q^k (1-q)^{N-k}} &= \sum_{n=-k}^\infty \binom{N+n-1}{N-k-1} q^n
	\end{align*}
	and
	\begin{align*}
		q^{-k} (1-q^{-1})^{N-k} &= \sum_{n=k}^N (-1)^{n-k} \binom{N-k}{n-k} q^{-n},
	\end{align*}
	we find that the constant term of $2/(N-k) (-1)^{N-k} q^{-N-k}$ is equal to the right-hand side of \eqref{eq:Nth-coeff-N}. Since this is zero, the desired claim~\eqref{eq:refinement} is proven.
	
	Finally, by combining \eqref{eq:constant-term} and \eqref{eq:Fourier-positive} with \eqref{eq:log-FourierN}, we obtain the Fourier series expansion
	\[
		\log \left(1 + \sum_{k=1}^\infty A_{S,N,\epsilon,k}(q) (2\sin \theta)^{2k} \right) = - 2 \sum_{k=1}^\infty \sum_{0 < m \in S} \frac{\epsilon^k q^{km}}{k} (\cos(2k\theta) -1),
	\]
	which is calculated as
	\begin{align*}
		&= -2 \sum_{k=1}^\infty \sum_{0 < m \in S} \frac{\epsilon^k q^{km}}{k} \sum_{j=1}^\infty \frac{(-1)^j}{(2j)!} (2k \theta)^{2j} = 2 \sum_{j=1}^\infty \frac{(-1)^{j-1}}{(2j)!} \left(\sum_{k=1}^\infty \sum_{0<m \in S} \epsilon^k k^{2j-1} q^{km} \right) (2\theta)^{2j}\\
		&= 2 \sum_{j=1}^\infty \frac{(-1)^{j-1}}{(2j)!} G_{S,N,\epsilon,2j}(q) (2\theta)^{2j}.
	\end{align*}
	Comparing it with \cref{def:Lehmer-poly} concludes the proof.
\end{proof}

\begin{remark}
	\cref{thm:ASNk-poly} generalizes Bachmann's generating series:
	\[
		1 + \sum_{k=1}^\infty A_k(q) X^{2k} = \exp \left(2\sum_{j=1}^\infty \frac{(-1)^{j-1}}{(2j)!} G_{2j}(q) \left(2 \arcsin \frac{X}{2}\right)^{2j} \right),
	\]
	where we put
	\begin{align}\label{eq:level-1-Eisen}
		G_k(q) \coloneqq G_{\{0\}, 1, 1, k}(q) = \sum_{n=1}^\infty \sigma_{k-1}(n) q^n = - \frac{B_k}{2k}(E_k(z) - 1).
	\end{align}
	It is important to highlight that Bachmann's result in~\cite{Bachmann2024} was established using the theory of multiple Eisenstein series. Additionally, Bachmann derived a similar generating series for $C_k(q)$ by utilizing ``odd multiple Eisenstein series".
\end{remark}

\begin{remark}
	Rose~\cite[Theorem 1.11 and 1.12]{Rose2015} computed the weight $2w$ part of the function $A_{S,N,\epsilon,k}(q)$ in terms of derivatives of theta functions. In contrast, our polynomials $\Lambda_k(x_2, x_4, \dots, x_{2k})$ provide simpler and unified expressions by assigning weight $kl$ to $x_k^l$. Notably, Rose's results require case distinctions depending on whether $S$ includes $0$, whereas our expression avoids such distinctions.
\end{remark}

\subsection{Rose's theorem on quasi-modularity}

As a direct consequence of \cref{thm:ASNk-poly}, we reproduce Rose's result. Moreover, although Larson~\cite{Larson2015} did not determine the group in the case $\epsilon = -1$, this also follows immediately. Since the generalized MacMahon function $A_{S,N,\epsilon,k}(q)$ is expressed as a polynomial in the Eisenstein series $G_{S,N,\epsilon,2j}(q)$, the problem reduces to identifying the modular properties of the Eisenstein series. 

\begin{proposition}\label{prop:Eisenstein-modular}
	For $\epsilon = 1$ and a positive even integer $k$, we have the following.
	\begin{enumerate}
		\item If $S = \{0\}$, then $G_{S, N, 1, k}(q) \in \widetilde{\calM}(\Gamma_0(N))$.
		\item If $N>1$ and $S = \{n \in \Z/N\Z : (n,N) = 1\}$, then $G_{S, N, 1, k}(q) \in \widetilde{\calM}_k(\Gamma_0(N))$.
		\item If $S = \{-l, l\}$ for some $l \in \Z/N\Z \setminus \{0\}$, then $G_{S, N, 1, k}(q) \in \widetilde{\calM}_k(\Gamma_1(N))$.
	\end{enumerate}
\end{proposition}

The assertion (3) corresponds to the results of Rose and Larson. To verify this proposition, we recall well-known claims.

\begin{lemma}[{\cite[Proposition 7.3.3]{CohenStromberg2017} and \cite[Exercise 1.2.8 (e)]{DiamondShurman2005}}]\label{lem:CS}
	We have the following.
	\begin{enumerate}
		\item[(a)] If $f \in \calM_k(\Gamma_0(N_1))$, then $f(N_2 z) \in \calM_k(\Gamma_0(N_1 N_2))$, and similarly for $\Gamma_1$ instead of $\Gamma_0$.
		\item[(b)] We have $E_{2,N}(z) \coloneqq E_2(z) - N E_2(Nz) \in \calM_2(\Gamma_0(N))$.
	\end{enumerate}
\end{lemma}

\begin{proof}[Proof of \cref{prop:Eisenstein-modular}]
	(1) Since $\sigma_{\{0\}, N, 1, k}(n) = \sigma_{\{0\}, 1, 1, k}(n/N)$, it follows that
	\[
		G_{\{0\},N,1,k}(q) = G_{\{0\}, 1, 1, k}(q^N) = - \frac{B_k}{2k} (E_k(Nz) - 1).
	\]
	Here, we put $\sigma_{S,N,\epsilon,k}(x) = 0$ for any $x \not\in\Z$. By \cref{lem:CS} (a), we obtain that $E_k(Nz) \in \calM_k(\Gamma_0(N))$ for $k \ge 4$. For $k=2$, \cref{lem:CS} (b) implies that 
	\begin{align}\label{eq:E2N-quasi}
		E_2(Nz) = \frac{1}{N}(E_2(z) - E_{2,N}(z)) \in \widetilde{\calM}_2(\Gamma_0(N)).
	\end{align}
	Therefore, for any positive even $k$, we conclude that $G_{\{0\}, N, 1, k}(q) \in \widetilde{\calM}(\Gamma_0(N))$.
	
	(2) By the inclusion-exclusion principle, we derive
	\[
		\sum_{\substack{d \mid n \\ (n/d, N)=1}} d^{k-1} = \sum_{l \mid N} \mu(l) \sum_{\substack{d \mid n \\ n/d \equiv 0\ (l)}} d^{k-1},
	\]
	which implies the equation
	\begin{align}\label{eq:GS=Moebius}
		G_{S, N, 1, k}(q) = -\frac{B_k}{2k} \sum_{l \mid N} \mu(l) (E_k(lz) - 1) = -\frac{B_k}{2k} \sum_{l \mid N} \mu(l) E_k(lz),
	\end{align}
	where we used the property of the M\"{o}bius function $\mu(n)$ that $\sum_{l \mid N} \mu(l) = 0$ for $N > 1$. Observing that $\widetilde{\calM}_k(\Gamma_0(l)) \subset \widetilde{\calM}_k(\Gamma_0(N))$ for $l \mid N$, we conclude that $G_{S,N,1,k}(q) \in \widetilde{\calM}_k(\Gamma_0(N))$. 
	
	(3) By applying~\cite[Theorem 4.2.3]{DiamondShurman2005} to the vector $v = (l, 0) \in (\Z/N\Z)^2$, for an even $k \ge 4$, we have
	\[
		\sum_{\substack{(m,n) \in \Z^2 \\ (m,n) \equiv (l, 0) \pmod{N}}} \frac{1}{(m \cdot Nz+n)^k} = c \cdot G_{\{-l, l\}, N, 1, k}(q)
	\]
	for a certain (explicit) constant $c \in \R$. Note that it is assumed in \cite[Section 4.2]{DiamondShurman2005} that $v \in (\Z/N\Z)^2$ is of order $N$. However, even if $v$ is not of order $N$, the same argument still holds. For instance, it is also helpful to refer to \cite[Chapter III-3, Proposition 22]{Koblitz1993}. For any $\gamma = \smat{a & b \\ c & d} \in \Gamma_1(N)$, it can be verified that
	\[
		\frac{1}{(cz+d)^k} \sum_{\substack{(m,n) \in \Z^2 \\ (m,n) \equiv (l, 0) \pmod{N}}} \frac{1}{(mN\frac{az+b}{cz+d}+n)^k} = \sum_{\substack{(m',n') \in \Z^2 \\ (m',n') \equiv (l, 0) \pmod{N}}} \frac{1}{(m'Nz+n')^k},
	\]
	where we changed variables via $(m', n') = (m, n) \smat{a & Nb \\ c/N & d}$. Hence, $G_{\{-l,l\}, N, 1, k}(q) \in \calM_k(\Gamma_1(N))$ for any even $k \ge 4$. For $k=2$, by similarly applying the argument in~\cite[Section 4.6]{DiamondShurman2005} to the vector $v = (l,0)$ and using the above reasoning, it can be verified that $G_{\{-l, l\}, N, 1, 2}(q) \in \widetilde{\calM}_2(\Gamma_1(N))$. Thus, the proof is complete.
\end{proof}

\begin{corollary}\label{cor:epsil1}
	For $\epsilon = 1$ and a positive integer $k$, the generalized MacMahon functions $A_{S,N, 1, k}(q)$ are quasi-modular forms in the following cases.
	\begin{enumerate}
		\item If $S = \{0\}$, then $A_{S,N,1,k}(q) \in \widetilde{\calM}(\Gamma_0(N))$.
		\item If $S = \{n \in \Z/N\Z : (n,N) = 1\}$, then $A_{S,N,1,k}(q) \in \widetilde{\calM}(\Gamma_0(N))$.
		\item If $S$ is symmetric, that is, $l \in S$ implies $-l \in S$, then $A_{S,N,1,k}(q) \in \widetilde{\calM}(\Gamma_1(N))$.
	\end{enumerate}
\end{corollary}

\begin{proof}
	The first two assertions immediately follow from \cref{thm:ASNk-poly} and \cref{prop:Eisenstein-modular}. As for the last assertion, note that any symmetric subset $S$ can be expressed as a disjoint union of subsets of the form $\{-l, l\}$ for some $0 \le l \le N/2$. This decomposition ensures that $G_{S,N,\epsilon,k}(q) \in \widetilde{\calM}(\Gamma_1(N))$, thereby establishing the desired result.
\end{proof}

For the case $\epsilon = -1$, it suffices to focus on the following equality:
\begin{align*}
	\sigma_{S,N,-1,k-1}(n) = 2^k \sigma_{S,N,1,k-1}(n/2) - \sigma_{S,N,1,k-1}(n).
\end{align*}
This implies that
\begin{align}\label{eq:epsil-1}
	G_{S,N,-1,k}(q) = 2^k G_{S,N,1,k}(q^2) - G_{S,N,1,k}(q)
\end{align}
and the following:

\begin{corollary}
	For $\epsilon = -1$ and a positive integer $k$, the generalized MacMahon functions $A_{S,N, -1, k}(q)$ are quasi-modular forms in the following cases.
	\begin{enumerate}
		\item If $S = \{0\}$, then $A_{S,N,-1,k}(q) \in \widetilde{\calM}(\Gamma_0(2N))$.
		\item If $S = \{n \in \Z/N\Z : (n,N) = 1\}$, then $A_{S,N,-1,k}(q) \in \widetilde{\calM}(\Gamma_0(2N))$.
		\item If $S$ is symmetric, then $A_{S,N,-1,k}(q) \in \widetilde{\calM}(\Gamma_1(2N))$.
	\end{enumerate}
\end{corollary}

\begin{proof}
	By combining \cref{lem:CS} and \eqref{eq:epsil-1} with the proof of \cref{prop:Eisenstein-modular}, we deduce that $G_{S,N,-1,k}(q)\in\widetilde{\calM}(\Gamma_0(2N))$ if $S=\{0\}$ or $S = \{n \in \Z/N\Z : (n,N) = 1\}$, and $G_{S,N,-1,k}(q)\in\widetilde{\calM}(\Gamma_1(2N))$ if $S = \{-l, l\}$ for some $l \in \Z/N\Z \setminus \{0\}$. Only the case $k=2$ in (3) requires further attention, so the details are explained below. Since $G_{\{-l,l\}, N,1,2}(q) \in \widetilde{\calM}_2(\Gamma_1(N))$, by the definition of quasi-modular forms, it can be expressed as
	\[
		G_{\{-l,l\}, N, 1, 2}(q) = g(z) + c E_2(z)
	\]
	for some $g \in \calM_2(\Gamma_1(N))$ and a constant $c \in \C$. By applying \cref{lem:CS} and \eqref{eq:E2N-quasi}, we have
	\[
		G_{\{-l,l\}, N,1,2}(q^2) = g(2z) + c E_2(2z) \in \widetilde{\calM}_2(\Gamma_1(2N)).
	\]
	The proof is then completed in the same way as in the proof of \cref{cor:epsil1}.
\end{proof}

The results above do not offer a full classification of all cases where $A_{S,N,\epsilon, k}(q)$ is a quasi-modular form. However, they do establish that all variants of MacMahon's original functions are quasi-modular forms.

\begin{corollary}
	MacMahon's functions $A_k(q), B_k(q), \dots, H_k(q)$ are quasi-modular forms as follows.
	\begin{align*}
		&A_k(q) \in \widetilde{\calM}(\SL_2(\Z)), & &B_k(q) \in \widetilde{\calM}(\Gamma_0(2)), & &C_k(q) \in \widetilde{\calM}(\Gamma_0(2)), & &D_k(q) \in \widetilde{\calM}(\Gamma_0(4)),\\
		&E_k(q) \in \widetilde{\calM}(\Gamma_1(5)), & &F_k(q) \in \widetilde{\calM}(\Gamma_1(10)), & &G_k(q) \in \widetilde{\calM}(\Gamma_1(5)), & &H_k(q) \in \widetilde{\calM}(\Gamma_1(10)).
	\end{align*}
\end{corollary}

\section{Proof of \cref{thm:level2}}\label{section:level2}

As a corollary of \cref{thm:ASNk-poly}, we obtain explicit formulas for the level 2 MacMahon partition functions $M_k^{(2)}(n)$ in terms of level 2 divisor sums defined by
\[
	\sigma_k^{(2)}(n) \coloneqq \sigma_{\{1\}, 2, 1, k}(n) = \sum_{\substack{d \mid n \\ n/d \equiv 1\ (2)}} d^k.
\]
In this section, we extend the formulas originally shown by MacMahon, recall Leli\`{e}vre's criteria--a fundamental strategy for obtaining prime-detecting expressions--for level $N$, and conclude with the proof of \cref{thm:level2}. For simplicity, we put
\[
	G_k^{(2)}(q) \coloneqq G_{\{1\}, 2, 1, k}(q) = \sum_{n=1}^\infty \sigma_{k-1}^{(2)}(n) q^n.
\]

\subsection{MacMahon's explicit formulas, revisited}

By \cref{thm:ASNk-poly}, the first few examples of $C_k(q) = A_{\{1\}, 2, 1, k}(q)$ are given as follows:
\begin{align*}
	C_1 &= G_2^{(2)},\\
	C_2 &= \frac{1}{12} \bigg(6 (G_2^{(2)})^2+G_2^{(2)} - G_4^{(2)} \bigg),\\
	C_3 &= \frac{1}{360} \bigg(60 (G_2^{(2)})^3+30 (G_2^{(2)})^2- 2(15 G_4^{(2)} -2) G_2^{(2)} -5 G_4^{(2)}+G_6^{(2)}\bigg),\\
	C_4 &= \frac{1}{20160} \bigg(840 (G_2^{(2)})^4+840 (G_2^{(2)})^3 -42(20 G_4^{(2)}-7) (G_2^{(2)})^2\\
		&\qquad + 4(14G_6^{(2)} - 105 G_4^{(2)} + 9)G_2^{(2)} +70 (G_4^{(2)})^2-49 G_4^{(2)}+14 G_6^{(2)} -G_8^{(2)}\bigg).
\end{align*}

These provide the following explicit formulas for $M_k^{(2)}(n)$. While the formulas for $M_1^{(2)}(n)$ through $M_3^{(2)}(n)$ were given by MacMahon, it is important to note that when considering up to $M_4^{(2)}(n)$, these are no longer $\Q[n]$-linear combinations of divisor sums but instead involve the Fourier coefficients of a cusp form.

\begin{proposition}\label{prop:MacMahon-explicit-2}
	We have
	\begin{align*}
		M_1^{(2)}(n) &= \sigma_1^{(2)}(n),\\
		M_2^{(2)}(n) &= \frac{1}{24} \bigg(\sigma_3^{(2)}(n) - (3n-2) \sigma_1^{(2)}(n)\bigg),\\
		M_3^{(2)}(n) &= \frac{1}{5760} \bigg(\sigma_5^{(2)}(n) - 5(3n-8) \sigma_3^{(2)}(n) + 2(15n^2 - 60n + 32) \sigma_1^{(2)}(n) \bigg),\\
		M_4^{(2)}(n) &= \frac{1}{16450560} \bigg(3\sigma_7^{(2)}(n) -119 (n-6) \sigma_5^{(2)}(n) + 357 (3n^2 - 30 n + 56) \sigma_3^{(2)}(n)\\
			&\qquad -51 (35n^3 - 420n^2 + 1176n -576) \sigma_1^{(2)}(n) + 14 a(n) \bigg),
	\end{align*}
	where
	\begin{align}\label{eq:Delta-2}
		\Delta_2(q) = \sum_{n=1}^\infty a(n) q^n &\coloneqq q \prod_{j=1}^\infty (1-q^j)^8 (1-q^{2j})^8 \in \calS_8(\Gamma_0(2)).
	\end{align}
\end{proposition}

\begin{proof}
	It is clear that $M_1^{(2)}(n) = \sigma_1^{(2)}(n)$ from $C_1(q) = G_2^{(2)}(q)$. Moreover, it is known that
	\begin{align*}
		\calM_2(\Gamma_0(2)) &= \C E_{2,2},\\
		\calM_4(\Gamma_0(2)) &= \C E_4 \oplus \C G_4^{(2)},\\
		\calM_6(\Gamma_0(2)) &= \C E_6 \oplus \C G_6^{(2)},\\
		\calM_8(\Gamma_0(2)) &= \C E_8 \oplus \C G_8^{(2)} \oplus \C \Delta_2,
	\end{align*}
	where $E_k(z)$ and $E_{2,2}(z)$ are defined in \eqref{eq:Eisenstein} and \cref{lem:CS}, respectively. Indeed, the dimensions of the spaces $\calM_k(\Gamma_0(2))$ ($k = 2,4,6,8$) are given as $1, 2, 2, 3$, respectively, by \cite[Theorem 3.5.1]{DiamondShurman2005}, and $\calS_8(\Gamma_0(2)) = \C \Delta_2$ is shown in~\cite[Proposition 3.2.2]{DiamondShurman2005}. The fact that $E_k$ and $G_k^{(2)}$ form a basis of the quotient space $\calM_k(\Gamma_0(2))/\calS_k(\Gamma_0(2))$ ($k \ge 4$) can be verified from the expression \eqref{eq:GS=Moebius} and their Fourier series expansions. By applying \cref{prop:quasi-structure} for a quasi-modular form $\phi = G_2^{(2)}$, we see that $C_k(q)$ can be expressed as a linear combination of the $D$-derivatives of these basis and $G_2^{(2)}$. For instance, since $C_2(q)$ is a linear combination of quasi-modular forms of weight $4$ and $2$, it can be written as
	\[
		C_2(q) = \bigg( a_1 E_4 + a_2 G_4^{(2)} + a_3 D E_{2,2} + a_4 DG_2^{(2)} \bigg) + \bigg(a_5 E_{2,2} + a_6 G_2^{(2)} \bigg)
	\]
	for some $a_1, \dots, a_6 \in \C$. By comparing the first few coefficients of the $q$-series expansions, a system of equations in $a_1, \dots, a_6$ is obtained, and solving this system yields
	\[
		C_2(q) = \frac{1}{24} G_4^{(2)} - \frac{1}{8} D G_2^{(2)} + \frac{1}{12} G_2^{(2)} = \frac{1}{24} \bigg(G_4^{(2)} - (3 D - 2) G_2^{(2)} \bigg),
	\]
	which implies the desired formula for $M_2^{(2)}(n)$. In a similar manner, the other results can be obtained.
\end{proof}

\subsection{Leli\`{e}vre's criteria}

We begin the proof of \cref{thm:level2}. The strategy, as in the work in~\cite{CraigIttersumOno2024, Gomez2024, Craig2024}, follows Leli\`{e}vre's criteria appeared in an unpublished note~\cite{Lelievre2004}. First, we extend this criteria to level $N$ for our purpose. Here, we consider the case in (2) of \cref{prop:Eisenstein-modular}, that is, $S = \{n \in \Z/N\Z : (n,N) = 1\}$ and 
\[
	G_k^{(N)}(q) \coloneqq G_{S,N,1,k}(q) = \sum_{n=1}^\infty \bigg(\sum_{\substack{d \mid n \\ (n/d, N)=1}} d^{k-1} \bigg)q^n.
\]

\begin{proposition}[Leli\`{e}vre's criteria]\label{prop:Lelievre}
	Let $k, l$ be positive integers with $l > k$. For $n \ge 2$, the $n$-th Fourier coefficient of
	\[
		f_{k,l}^{(N)}(q) \coloneqq (D^l + 1) G_{k+1}^{(N)}(q) - (D^k + 1) G_{l+1}^{(N)}(q)
	\]
	satisfies
	\[
		\begin{cases}
			=0 &\text{if $n$ is prime with $n \nmid N$},\\
			<0 &\text{if all prime factors $p$ of $n$ satisfy $p \mid N$},\\
			>0 &\text{otherwise}.
		\end{cases}
	\]
\end{proposition}

\begin{proof}
	By the definition of $G_{k}^{(N)}(q)$, we have
	\[
		f_{k,l}^{(N)}(q) = \sum_{n=1}^{\infty} \bigg(\sum_{\substack{d \mid n\\(n/d,N)=1}}			(n^{l}+1)d^{k}-(n^{k}+1)d^{l}\bigg)q^{n}.
	\]
	
	If $n$ is prime and $n\nmid N$, then the $n$-th Fourier coefficient of $f_{k,l}^{(N)}(q)$ is given by
	\[
		(n^{l}+1)-(n^{k}+1)+(n^{l}+1)n^{k}-(n^{k}+1)n^{l}=0.
	\]
	
	If all prime factors $p$ of $n$ satisfy $p\mid N$, the contribution to the sum comes only from the case where $d=n$, and the $n$-th Fourier coefficient of $f_{k,l}^{(N)}(q)$ becomes
	\[
		(n^{l}+1)n^{k}-(n^{k}+1)n^{l}=n^{k}-n^{l}<0.
	\]
	
	Finally, in the remaining case, we can uniquely factor $n$ as $n=n_{1}n_{2}$, where $(n_{2},N)=1$, $n_2 > 1$, and all primes factors of $n_{1}$ divide $N$. For any divisor $d \mid n$ such that $(n/d, N) =1$ except for $d = n$, since $d \le n/2$, we have
	\begin{align*}
		(n^l + 1)d^k - (n^k+1) d^l &\ge (n^l+1)d^k - 2n^k d^k \left(\frac{n}{2}\right)^{l-k}\\
			&= d^k \bigg(1 + \bigg(1- \frac{1}{2^{l-k-1}}\bigg) n^l\bigg) > 0.
	\end{align*}
	The term corresponding to $d=n$ gives $n^k - n^l < 0$, which requires careful consideration. However, by considering it together with the term corresponding to $d = n_1$, we obtain the following estimate:
	\begin{align}\label{eq:easy-estimate}
	\begin{split}
		&(n^l+1)n^k - (n^k+1)n^l + (n^l+1)n_1^k - (n^k+1)n_1^l\\
		&\quad = n_1^k \bigg(n_2^k (n_2^{l-k}-1) n_1^{l-k} (n_1^k-1) - (n_2^k+1)(n_1^{l-k}-1) \bigg).
	\end{split}
	\end{align}
	If $n_1 = 1$, this equals $0$. In this case, by the assumption made in the case distinction, $n_2$ must be a composite number, implying that there is a divisor of $n$ other than $d = n_1, n$. Therefore, the $n$-th coefficient is positive. 
	If $n_1 > 1$ and $l - k \ge 2$, since $n_1^{l-k}(n_1^k-1) > n_1^{l-k}-1$, \eqref{eq:easy-estimate} is bounded by
	\[
		> n_1^k \bigg(n_2^k (n_2^{l-k}-1) (n_1^{l-k}-1) - (n_2^k+1)(n_1^{l-k}-1) \bigg) = n_1^k (n_1^{l-k}-1) (n_2^k (n_2^{l-k} - 2) - 1) > 0.
	\]
	Finally, if $n_1 > 1$ and $l - k = 1$, \eqref{eq:easy-estimate} becomes
	\begin{align*}
		&= n_1^k \bigg(n_2^k (n_2-1) n_1 (n_1^k-1) - (n_2^k+1)(n_1-1) \bigg)\\
		&\ge n_1^k (n_1-1) \bigg(n_2^k (n_1 + n_1^2 + \cdots + n_1^k) - (n_2^k + 1) \bigg) > 0,
	\end{align*}
	which concludes the proof. 
\end{proof}

\begin{proof}[Proof of \cref{thm:level2}]
	By \cref{prop:MacMahon-explicit-2}, we can easily verify that
	\begin{align*}
		f_{1,3}^{(2)}(q) &= (D+1) \bigg((D^2 - 4D + 3) C_1(q) - 24 C_2(q) \bigg),\\
		f_{1,5}^{(2)}(q) &= (D+1) \bigg((D^4 - D^3 -14D^2 + 29D - 15) C_1(q) - 120 (3D-8)C_2(q) - 5760 C_3(q) \bigg).
	\end{align*}
	\cref{prop:Lelievre} immediately implies the desired results.
\end{proof}

\subsection{Remarks on further prime-detecting expressions}

In the proof of \cref{thm:level2} above, the key idea was to realize the function $f_{k,l}^{(2)}(q)$ of Leli\`{e}vre's criteria through the combination of $C_1(q)$, $C_2(q)$, and $C_3(q)$. The same approach for $C_4(q)$, however, fails because the cusp form $\Delta_2$ cannot be expressed as a linear combination of $D$-derivatives of Eisenstein series. A similar obstruction arose in Craig--van Ittersum--Ono's work~\cite{CraigIttersumOno2024} due to $\Delta \in \calS_{12}(\SL_2(\Z))$. In contrast, the case of $A_k(q)$ benefits from the fact that $\calS_{14}(\SL_2(\Z)) = \emptyset$, which allows the derivation of the fifth prime-detecting expression. For these reasons, in the level $2$ case, if we make a reasonable conjecture based on previous work~\cite{CraigIttersumOno2024}, it appears that the level 2 prime-detecting expressions described using $M_k^{(2)}(n)$ are essentially limited to the two provided in \cref{thm:level2}.

By following a similar approach, the level 3 version can also be considered. Let $M_k^{(3)}(n)$ denote the $n$-th coefficient of $q$-series expansion of $A_{\{1,2\}, 3, 1, k}(q)$, that is,
\[
	A_{\{1,2\}, 3, 1, k}(q) \coloneqq \sum_{n=1}^\infty M_k^{(3)}(n) q^n \coloneqq \sum_{\substack{0 < m_1 < m_2 < \cdots < m_k \\ m_i \not\equiv 0\ (3)}} \frac{q^{m_1+m_2+\cdots+m_k}}{(1-q^{m_1})^2 (1-q^{m_2})^2 \cdots (1-q^{m_k})^2}.
\]
A level 3 prime-detecting expression is obtained as
\begin{align}
	(n^2 - 3n + 2) M_1^{(3)}(n) -12 M_2^{(3)}(n) \begin{cases}
		=0 &\text{if $n$ is prime except for $3$},\\
		<0 &\text{if $n = 3^l$ for $l \in \Z_{\ge 1}$},\\
		>0 &\text{otherwise}.
	\end{cases}
\end{align}
However, due to the cusp form
\[
	\Delta_3(q) \coloneqq q \prod_{j=1}^\infty (1-q^j)^6 (1-q^{3j})^6 \in \calS_6(\Gamma_0(3)),
\]
the argument stops here. While the quasi-modularity of the generalized MacMahon functions holds for more general levels $N$, the phenomenon of prime-detecting expressions involving the MacMahon function seems to be specific to the lower levels.

Finally, applying the same argument to the MacMahon function $B_k(q)$ would require examining when
\[
	\sum_{d \mid n} (-1)^d \bigg((n^l+1) d^k - (n^k+1) d^l \bigg)
\]
becomes zero and how its sign changes. However, due to the influence of the factor $(-1)^d$, a much more detailed and refined analysis is necessary. For this reason, it is not addressed in this article.

\section{Proof of \cref{thm:lattice}}\label{section:theta}

The proof of \cref{thm:lattice} involves the use of theta functions, but the foundation remains Leli\`{e}vre's criteria. First, we show that the counting functions $r_L(n)$ for the lattice $L \in \{L_1, L_2, E_8\}$ can be expressed using divisor functions.

\begin{proposition}
	For any $n \ge 1$, we have $r_{L_1}(n) = 4 \sigma_1^{(1,4)}(n)$, $r_{L_2}(n) = 4 \sigma_1^{(3,4)}(n)$, and $r_{E_8}(2n) = 240 \sigma_3(n)$, where
	\[
		\sigma_k^{(a,N)}(n) \coloneqq \begin{cases}
			\sigma_k(n) &\text{if } n \equiv a \pmod{N},\\
			0 &\text{otherwise},
		\end{cases}
	\]
	for integers $N > 0$ and $0 \le a < N$.
\end{proposition}

\begin{proof}
	To begin, the function $r_{L_1}(n)$ is expressed as
	\[
		r_{L_1}(n) = \# \left\{(n_1, n_2, n_3, n_4) \in \Z^4 : 4 \left(n_1^2 + n_2^2 + \frac{n_3(n_3+1)}{2} + \frac{n_4(n_4+1)}{2}\right) + 1 = n \right\}.
	\]
	From this expression, if $n \not\equiv 1 \pmod{4}$, then $r_{L_1}(n) = 0$. For $n \equiv 1 \pmod{4}$, it is known that $r_{L_1}(n) = 4\sigma_1(n)$, (see, for instance,~\cite[(1.16)]{Hirschhorn2005}). Therefore, $r_{L_1}(n) = 4\sigma_1^{(1,4)}(n)$. Similarly, it is also known that $r_{L_2}(n) = 4 \sigma_1^{(3,4)}(n)$, (see, for instance,~\cite[(3)]{Melham2008}). 	
	As for the $E_8$-lattice, it is known that the theta function defined by
	\[
		\Theta_{E_8}(z) \coloneqq \sum_{x \in E_8} q^{\|x\|^2/2}
	\]
	equals $E_4(z)$, that is, $r_{E_8}(2n) = 240\sigma_3(n)$, (see~\cite[p.122]{ConwaySloane1999}). 
\end{proof}

\begin{proof}[Proof of \cref{thm:lattice}]
	The claims are equivalent to
	\[
		(n^3 +1) \sigma_1^{(a,4)}(n) - (n+1) \sigma_3(n) \begin{cases}
			=0 &\text{if $n$ is prime with $n \equiv a \pmod{4}$},\\
			<0 &\text{if $n \not\equiv a \pmod{4}$},\\
			>0 &\text{otherwise},
		\end{cases}
	\]
	for $a = 1,3$. This follows from a simple observation and the application of \cref{prop:Lelievre} with $N=1$ and $(k,l) = (1,3)$
\end{proof}

\begin{remark}
	Leli\`{e}vre's criteria for general arithmetic progressions modulo $N$ were discussed in Gomez~\cite{Gomez2024}.
\end{remark}

\bibliographystyle{amsalpha}
\bibliography{References} 

\begin{bibdiv}
\begin{biblist}

\bib{AndrewsRose2013}{article}{
      author={Andrews, George~E.},
      author={Rose, Simon C.~F.},
       title={Mac{M}ahon's sum-of-divisors functions, {C}hebyshev polynomials,
  and quasi-modular forms},
        date={2013},
     journal={J. Reine Angew. Math.},
      volume={676},
       pages={97\ndash 103},
}

\bib{Bachmann2024}{article}{
      author={Bachmann, Henrik},
       title={Macmahon's sums-of-divisors and their connection to multiple
  {E}isenstein series},
        date={2024},
     journal={Res. Number Theory},
      volume={10},
      number={2},
       pages={Paper No. 50, 10},
}

\bib{CohenStromberg2017}{book}{
      author={Cohen, Henri},
      author={Str\"{o}mberg, Fredrik},
       title={Modular forms},
      series={Graduate Studies in Mathematics},
   publisher={American Mathematical Society, Providence, RI},
        date={2017},
      volume={179},
        note={A classical approach},
}

\bib{ConwaySloane1999}{book}{
      author={Conway, J.~H.},
      author={Sloane, N. J.~A.},
       title={Sphere packings, lattices and groups},
     edition={Third},
      series={Grundlehren der mathematischen Wissenschaften [Fundamental
  Principles of Mathematical Sciences]},
   publisher={Springer-Verlag, New York},
        date={1999},
      volume={290},
        note={With additional contributions by E. Bannai, R. E. Borcherds, J.
  Leech, S. P. Norton, A. M. Odlyzko, R. A. Parker, L. Queen and B. B. Venkov},
}

\bib{Craig2024}{misc}{
      author={Craig, William},
       title={Higher level $q$-multiple zeta values with applications to
  quasimodular forms and partitions},
        date={2024},
         url={https://arxiv.org/abs/2409.13874},
        note={arXiv:2409.13874},
}

\bib{CraigIttersumOno2024}{article}{
      author={Craig, William},
      author={van Ittersum, Jan-Willem},
      author={Ono, Ken},
       title={Integer partitions detect the primes},
        date={2024},
     journal={Proc. Natl. Acad. Sci. USA},
      volume={121},
      number={39},
       pages={Paper No. e2409417121},
}

\bib{DiamondShurman2005}{book}{
      author={Diamond, Fred},
      author={Shurman, Jerry},
       title={A first course in modular forms},
      series={Graduate Texts in Mathematics},
   publisher={Springer-Verlag, New York},
        date={2005},
      volume={228},
}

\bib{Gomez2024}{misc}{
      author={Gomez, Kevin},
       title={Mac{M}ahonesque partition functions detect sets related to
  primes},
        date={2024},
         url={https://arxiv.org/abs/2409.14253},
        note={arXiv:2409.14253},
}

\bib{Hirschhorn2005}{article}{
      author={Hirschhorn, Michael~D.},
       title={The number of representations of a number by various forms},
        date={2005},
     journal={Discrete Math.},
      volume={298},
      number={1-3},
       pages={205\ndash 211},
}

\bib{KanekoZagier1995}{incollection}{
      author={Kaneko, Masanobu},
      author={Zagier, Don},
       title={A generalized {J}acobi theta function and quasimodular forms},
        date={1995},
   booktitle={The moduli space of curves ({T}exel {I}sland, 1994)},
      series={Progr. Math.},
      volume={129},
   publisher={Birkh\"{a}user Boston, Boston, MA},
       pages={165\ndash 172},
}

\bib{Koblitz1993}{book}{
      author={Koblitz, Neal},
       title={Introduction to elliptic curves and modular forms},
     edition={Second},
      series={Graduate Texts in Mathematics},
   publisher={Springer-Verlag, New York},
        date={1993},
      volume={97},
}

\bib{Larson2015}{article}{
      author={Larson, Hannah},
       title={Proof of conjecture regarding the level of {R}ose's generalized
  sum-of-divisor functions},
        date={2015},
     journal={Res. Number Theory},
      volume={1},
       pages={Paper No. 16, 8},
}

\bib{Lehmer1966}{article}{
      author={Lehmer, D.~H.},
       title={Some properties of the cyclotomic polynomial},
        date={1966},
     journal={J. Math. Anal. Appl.},
      volume={15},
       pages={105\ndash 117},
}

\bib{Lelievre2004}{misc}{
      author={Leli\`{e}vre, Samuel},
       title={Quasimodular forms with {F}ourier coefficients zero at primes},
        date={2004},
        note={unpublished note,
  \url{https://www.imo.universite-paris-saclay.fr/~samuel.lelievre/papers/qmp/qmp.pdf}},
}

\bib{MacMahon1920}{article}{
      author={MacMahon, P.~A.},
       title={Divisors of numbers and their continuations in the theory of
  partitions},
        date={1920},
     journal={Proc. London Math. Soc. (2)},
      volume={19},
      number={1},
       pages={75\ndash 113},
}

\bib{MOS1966}{book}{
      author={Magnus, Wilhelm},
      author={Oberhettinger, Fritz},
      author={Soni, Raj~Pal},
       title={Formulas and theorems for the special functions of mathematical
  physics},
     edition={enlarged},
      series={Die Grundlehren der mathematischen Wissenschaften, Band 52},
   publisher={Springer-Verlag New York, Inc., New York},
        date={1966},
}

\bib{MatsusakaShibukawa2024}{article}{
      author={Matsusaka, Toshiki},
      author={Shibukawa, Genki},
       title={Curious congruences for cyclotomic polynomials {II}},
        date={2024},
     journal={Res. Number Theory},
      volume={10},
      number={1},
       pages={Paper No. 3, 9},
}

\bib{Melham2008}{article}{
      author={Melham, R.~S.},
       title={Analogues of two classical theorems on the representations of a
  number},
        date={2008},
     journal={Integers},
      volume={8},
       pages={A51, 10},
}

\bib{OnoSingh2024}{article}{
      author={Ono, Ken},
      author={Singh, Ajit},
       title={Remarks on {M}ac{M}ahon's {$q$}-series},
        date={2024},
     journal={J. Combin. Theory Ser. A},
      volume={207},
       pages={Paper No. 105921, 9},
}

\bib{Rose2015}{article}{
      author={Rose, Simon C.~F.},
       title={Quasi-modularity of generalized sum-of-divisors functions},
        date={2015},
     journal={Res. Number Theory},
      volume={1},
       pages={Paper No. 18, 11},
}

\bib{Viazovska2017}{article}{
      author={Viazovska, Maryna~S.},
       title={The sphere packing problem in dimension 8},
        date={2017},
     journal={Ann. of Math. (2)},
      volume={185},
      number={3},
       pages={991\ndash 1015},
}

\bib{Zagier2008}{incollection}{
      author={Zagier, Don},
       title={Elliptic modular forms and their applications},
        date={2008},
   booktitle={The 1-2-3 of modular forms},
      series={Universitext},
   publisher={Springer, Berlin},
       pages={1\ndash 103},
}

\end{biblist}
\end{bibdiv}

\end{document}